\documentclass[fullpage]{amsart}
\usepackage{amssymb,amstext,amsmath,latexsym,amsthm, amscd, mathtools}
\usepackage{graphicx}
\DeclareGraphicsExtensions{.jpg}
\newcommand\C{\mathbb{C}}
\newcommand\R{\mathbb{R}}
\newcommand\Q{\mathbb{Q}}
\newcommand\Z{\mathbb{Z}}

\newcommand\field{K}
\newcommand\ring{\mathcal{O}}
\newcommand\discr{D_K}
\newcommand\diff{\mathfrak{d}}

\newcommand\primep{\mathfrak{p}}

\newcommand\ideala{\mathfrak{a}}
\newcommand\idealb{\mathfrak{b}}
\newcommand\idealn{\mathfrak{n}}
\newcommand\degr{r}
\theoremstyle{plain}
\newtheorem{thm}{Theorem}[section]

\newtheorem{lem}[thm]{Lemma}

\theoremstyle{definition}

\theoremstyle{remark}
\newtheorem{rem}[thm]{Remark}
\numberwithin{equation}{section}
\title{$p$-adic measures for reciprocals of $L$-functions of totally real number fields}
\date{}
\author{Razan Taha}
\begin{document}
\begin{abstract}
	We construct $p$-adic measures which interpolate the special values of reciprocals of $p$-adic $L$-functions of totally real number fields $K$ at negative integers. These measures are defined by analyzing the non-constant term of partial Eisentein series of the Hilbert modular group of $K$. 
\end{abstract}
\maketitle
\tableofcontents
\section{Introduction}
Let $K$ be a totally real number field of finite degree over $\Q$ and let $K(p^\infty)$ be  the maximal abelian unramified-outside-$p$ extension of $K.$ The goal of this paper is to construct a $p-$adic measure $\lambda$ on the Galois group $G\left(K(p^\infty)/K\right)$ from the non-constant terms of the Fourier expansion of holomorphic Eisenstein series on $SL_2(\ring),$ where $\ring$ is the ring of integers of $K$. Moreover, the Mellin transform of $\lambda$ turns out to be the reciprocal of the $p-$adic $L-$function of the totally real field $K$. This $L$-function is the $p$-adic analogue of the classical complex valued Hecke $L$-function and has several different constructions which we will briefly describe.

The study of $p$-adic $L$-functions starts with the fact that the values of the Riemann zeta function at negative integers are rational and can be expressed in terms of Bernoulli numbers as
\begin{equation*}
\zeta(1-k)=(-1)^{k+1}\frac{B_k}{k}, \,\,\, \forall k>0.
\end{equation*} 
Kummer \cite{Kummer} proved that these values satisfy special congruences modulo powers of a prime number p: given two even integers $k$ and $k'$ which are not divisible by $p-1$ and for all $d>0,$ if $k \equiv k' \bmod p^{d-1}(p-1)$ then
\begin{equation}
(1-p^{k-1})\zeta(1-k) \equiv (1-p^{k'-1})\zeta(1-k') \bmod p^d.
\end{equation}
Kubota and Leopoldt \cite{KL} were the first to interpret the Kummer congruences as a $p$-adic function which interpolates the values of the Riemann zeta function at negative integers. They proved that there exists a continuous function $L_p$ on $\Z_p$ such that for all positive integers $k$ which are congruent to $0 \bmod p-1,$
\begin{equation} \label{eq14}
L_p(\chi, 1-k) = L(\chi,1-k)(1-\chi(p)p^{k-1})
\end{equation}
where the right hand side of (\ref{eq14}) is the Dirichlet $L$-function with the Euler factor at $p$ removed. 

For a totally real field $\field,$ the Hecke $L$-function attached to a character $\chi$ with conductor $\mathfrak{n}$ is given by 
\begin{equation}
L(s,\chi)=\sum_{(\mathfrak{a},\mathfrak{n})=1}\frac{\chi(\mathfrak{a})}{N(\mathfrak{a})^s}  = 
\sum_{[\mathfrak{a}]}\chi([\mathfrak{a}])\sum_{\mathfrak{b} \in [\mathfrak{a}]}\frac{1}{N(\mathfrak{b})^s}
\end{equation}
where $[\mathfrak{a}]$ denotes the class of the ideal $\mathfrak{a}.$
Siegel \cite{Siegel} proved that the values of the partial $L$-function $\sum_{\mathfrak{b} \in [\mathfrak{a}]}N(\mathfrak{b})^{k-1}$ are rational, which led to several different constructions of the $p$-adic $L$-function for a totally real field. Serre \cite{Serre} constructed $p$-adic modular forms and obtained the $p$-adic $L$-function as the constant term of a $p$-adic Eisenstein series. Shintani \cite{Shintani} proved an elementary formula for this rational number in terms of certain generalized Bernoulli numbers, which Cassou-Nogues \cite{CN} then used to define the p-adic L-function in terms of generalized Bernoulli numbers. Coates and Sinnott \cite{CS} equivalently proved the generalized Kummer congruences for real quadratic fields and Deligne and Ribet \cite{DR} proved them in general for any totally real field. 

One long term goal of this project is to study $p$-adic $L$-functions using methods inspired by the Langlands program. In the complex setting, the Langlands-Shahidi method establishes the analytic continuation and functional equations of many $L$-functions by studying the non-constant terms of the Fourier expansion of Eisenstein series on reductive groups. In particular, the denominators of these Fourier coefficients contain suitable $L$-functions. This inspired Gelbart, Miller, Panchishkin, and Shahidi to introduce an analog of parts of the Langlands-Shahidi method to the $p$-adic setting in \cite{GMPS}. They proved the existence of a $p$-adic measure $\mu^*$ which can be expressed in terms of the non-zero Fourier coefficients of classical holomorphic Eisenstein series such that its Mellin transform is the reciprocal of the $p$-adic Riemann $\zeta$-function.

In this paper, we generalize the result of \cite{GMPS} to the case of totally real number fields. We now state our main theorem:
\begin{thm} \label{thm1}
	Let $p \in \Q$ be an odd prime number such that $p$ does not divide $[K:\Q]$ or the class number of $K(e^{2\pi i/p})$. Assume further that no prime $\wp$ of the field $K(e^{2\pi i /p}+e^{-2\pi i/p})$ lying above $p$ splits in $K(e^{2\pi i/p}).$ Let $\primep$ be a prime ideal of $\ring$ lying above $p$ and $h^+$ the strict class number of $K.$ 
	\begin{enumerate}
		\item There exists a $p$-adic measure $\lambda$ on $G\left(K(p^\infty)/K\right)$ whose Mellin transform is the reciprocal of the $p$-adic $L$-function of the totally real field $K$
		\begin{equation*}
		\int_{G\left(K(p^\infty)/K\right)} N(x)^{k-1}\, d\lambda = h^+ (1-N(\primep)^{k-1})^{-1} \zeta(1-k)^{-1}
		\end{equation*}
		for all even positive integers $k.$ 
		\item For any non-trivial Hecke character of finite type $\chi \mod \primep^m$ with $m>0$ , we have
		\begin{equation*}
		\int_{G\left(K(p^\infty)/K\right)}\chi(x) N(x)^{k-1}\, d\lambda = h^+ L(1-k,\chi)^{-1}
		\end{equation*} 
		for any positive integer $k$ satisfying the parity condition $\chi(-1)=(-1)^k$.
	\end{enumerate}
Moreover, the measure $\lambda$ can be expressed in terms of Fourier coefficients of the Eisenstein series defined by $\chi$ on the Hilbert upper half plane. 
\end{thm}
\begin{rem}
	The conditions on $p$ are those that ensure that Kummer's criterion holds for $K$ (\cite{Greenberg}). Greenberg mentions in the introduction to {\it loc. cit.} that the condition $p \nmid [K:\Q] $ may be relaxed. However, the other conditions are the regularity conditions on $p$ and cannot be removed. 
\end{rem}
The paper is organized as follows: In Section \ref{sectionHMF}, we give a brief overview of Hilbert modular forms in the classical setting and we define holomorphic Hilbert Eisenstein series. In Section \ref{FCES}, we review the Fourier expansion of holomorphic Eisenstein series of weight $k$ and level $\mathfrak{n}.$ In Section \ref{construction}, we use the non-constant terms of this Fourier expansion to construct a $p$-adic distribution $\lambda.$ This is the main construction of the paper. Moreover, we show that $\lambda$ satisfies the properties given in Theorem \ref{thm1}. In Section \ref{proof}, we prove that $\lambda$ is in fact a $p$-adic measure.
\newline \newline
{\bf Acknowledgement.} The author would like to thank Andrei Jorza, Ellen Eischen, and A. Raghuram for their helpful conversations. She would also like to thank Romyar Sharifi for his explanation of the structure of the Iwasawa algebra. This work is partially supported by Freydoon Shahidi's NSF grants DMS-1500759 and DMS-1801273 as a part of the thesis written under him at Purdue University.  
\section{Hilbert Modular Forms} \label{sectionHMF}
In this section, we set up the main definitions and results that we need from the theory of Hilbert modular forms. A thorough treatment of the subject can be found in \cite{Freitag}, for example. 
\subsection{Structure of the Hilbert Modular Group}
Let $\field$ be a totally real number field of degree $\degr$ over $\Q.$ Let $\ring$ denote the ring of integers of $\field,$ $\diff$ its different ideal, $\discr$ its discriminant, and $h$ its class number. Let $N:\field \rightarrow \Q$ denote the norm map on $\field$. 

We fix an ordering of the $r$ embeddings $K \hookrightarrow \R.$ For any $\alpha \in \field,$ we denote by $\alpha^{(1)},\cdots, \alpha^{(\degr)}$ the conjugates of $\alpha.$ We say $\alpha$ is totally positive, denoted $\alpha \gg 0,$ if $\alpha^{(i)} \geq 0$ for all $1 \leq i \leq \degr.$

The Hilbert modular group $SL_2(\ring)$ acts discontinuously on the product of $\degr$ copies of the upper half plane $\mathfrak{h}$ by Mobius transformations componentwise: for any $A=
\begin{psmallmatrix} a & b \\ c & d \end{psmallmatrix} \in SL_2(\ring)$ and any $z=(z_1,\cdots, z_{\degr}) \in \mathfrak{h}^{\degr},$
\begin{equation}
A \cdot z = \left(\frac{a^{(1)}z_1+b^{(1)}}{c^{(1)}z_1+d^{(1)}}, \cdots, \frac{a^{(\degr)}z_{\degr}+b^{(\degr)}}{c^{(\degr)}z_{\degr}+d^{(\degr)}}\right).
\end{equation}
This produces an infinite number of cusps but only a finite number of inequivalent cusps.
\begin{thm}{(\cite{Freitag} pg. 37 Corollary $3.5_1$)}
The equivalence classes of cusps $\kappa$ under the action of $SL_2(\ring)$ are in one to one correspondence with the ideal classes of $\field.$ In particular, the number of inequivalent cusps is $h$. 	
\end{thm} 
The correspondence is easy to describe: the cusp represented by $(a:b) \in \mathbb{P}^1(K)$ is mapped to the class of the ideal $\langle a,b \rangle .$ This implies that the cusp $\kappa=\infty$ which is represented by $(1:0)\in \mathbb{P}^1(K)$ corresponds to the trivial ideal class consisting of all principal ideals of $\field$. 
\subsection{Modular Forms on Congruence Subgroups}
For $k=(k_1, \cdots, k_\degr) \in \Z^\degr, \, z=(z_1, \cdots, z_\degr) \in \C^\degr,$ and $c,d \in \field,$ we write 
\begin{align}
Tr(k) = \sum_{i=1}^{\degr} k_i \qquad &\qquad  Tr(z) = \sum_{i=1}^{\degr} z_i \\
(cz+d)^k &=  \prod_{i=1}^{\degr} (c^{(i)}z_i+d^{(i)})^{k_i}
\end{align}
The {\it slash operator} $\mid_k$ is defined for a complex valued function $f$ on $\mathfrak{h}^\degr$ and an element $\begin{psmallmatrix} * & * \\ c & d \end{psmallmatrix} \in SL_2(\ring)$ by 
\begin{equation}
(f \mid_k \begin{psmallmatrix} * & * \\ c & d \end{psmallmatrix})(z)=(cz+d)^{-k} f(\begin{psmallmatrix} * & * \\ c & d \end{psmallmatrix} \cdot z).
\end{equation} 
Given a congruence subgroup $\Gamma \subset SL_2(\ring),$ we say that $f$ is a Hilbert modular form of weight $k=(k_1, \cdots, k_\degr)$ and level $\Gamma$ if $f \mid_k \gamma =f$ for all $\gamma \in \Gamma.$ We will only be interested in the case when $k_1=k_2=\cdots=k_\degr=k$ and we will say the Hilbert modular form is of (parallel) weight $k$. One basic example of Hilbert modular forms is the Eisenstein series attached to a character of $\field^*$. We denote the space of all the Hilbert modular forms of weight $k$ and level $\Gamma$ by $\mathcal{M}_k(\Gamma).$ 

We will be most interested in the congruence subgroups
\begin{align} \label{eq10}
\Gamma_0(\mathfrak{n}) &= \left \{ \gamma = 
\begin{pmatrix}
a & b \\ c & d 
\end{pmatrix}
\in SL_2(\ring) \, \middle| \, \gamma \equiv 
\begin{pmatrix}
* & * \\ 0 & * 
\end{pmatrix}
\bmod \mathfrak{n}
\right \} \\  \label{eq11}
\Gamma_1(\mathfrak{n}) &= \left \{ \gamma = 
\begin{pmatrix}
a & b \\ c & d 
\end{pmatrix}
\in SL_2(\ring) \, \middle| \, \gamma \equiv 
\begin{pmatrix}
1 & * \\ 0 & 1 
\end{pmatrix}
\bmod \mathfrak{n}
\right \} 
\end{align}
of $SL_2(\ring),$ where $\mathfrak{n}$ is an integral ideal of $\field.$ 

If $f \in \mathcal{M}_k(\Gamma_0(\idealn)) $ or $f \in \mathcal{M}_k(\Gamma_1(\idealn)), $ it has a Fourier expansion at the cusp $\infty$ of the form
\begin{equation}
	f(z)=\sum_{\substack{\nu \equiv 0 \bmod (\diff\idealn)^{-1} \\ \nu \gg 0}} a(\nu)\exp(\nu z)
\end{equation}
where $\exp(\nu z)=e^{2\pi i Tr(\nu z) }.$

For any general cusp $\kappa,$ there is a matrix $A=\begin{psmallmatrix} \alpha_1 & \alpha_2 \\ \alpha_3 & \alpha_4 \end{psmallmatrix}$ of determinant 1 such that $A\kappa = \infty$ (\cite{Klingen} pg. 181). Moreover, $A$ can be chosen so that
\begin{equation} 
\kappa = - \frac{\alpha_4}{\alpha_3}, \quad \mathfrak{b}=\langle \alpha_3,\alpha_4 \rangle , \quad \mathfrak{b}^{-1}=\langle \alpha_1,\alpha_2 \rangle , \quad \gcd (\mathfrak{n},\mathfrak{b})=1.
\end{equation}
In this case, $f$ admits a Fourier expansion at $\kappa=A^{-1}\infty$ of the form 
\begin{equation} \label{eq7}
(\alpha_3 A^{-1}z+\alpha_4)^{k}f(A^{-1}z)=\sum_{\substack{\nu \equiv 0 \bmod \idealb^2(\diff\idealn)^{-1} \\ \nu \gg 0}} a_A(\nu)\exp(\nu z).
\end{equation}

\subsection{Eisenstein Series}

Recall that $\mathfrak{n}$ is an integral ideal of $\field.$ A grossencharacter $\chi$ is said to be of finite type with conductor $\mathfrak{n},$ denoted $\chi \bmod \mathfrak{n},$ whenever $\chi$ is a character $\left(\ring/\mathfrak{n}\right)^* \to \C^*$. 

Given a grossencharacter $\chi \bmod \mathfrak{n}$ of finite type, an integer $k \geq 3, $ and a cusp $\kappa$ of $\Gamma_0(\mathfrak{n}),$ the weight $k$ holomorphic Eisenstein series for $\Gamma_0(\mathfrak{n})$ at $\kappa$ is given by 
\begin{equation} \label{eq12}
E_{k,\mathfrak{n}}(\chi,z, \kappa) = \sum_{\gamma \in \Gamma_0(\mathfrak{n})_{\kappa} \backslash \Gamma_0(\mathfrak{n}) } \text{Im}(\gamma z)^k
\end{equation} 
for any $z \in \mathfrak{h}^r.$ Here, for any congruence subgroup $\Gamma,$ $\Gamma_{\kappa}$ denotes the stabilizer of $\kappa$ inside $\Gamma.$ 

If $E_{k,\mathfrak{n}}^{\Gamma_1}(\chi,z, \kappa_1)$ denotes the weight $k$ holomorphic Eisenstein series for $\Gamma_1(\mathfrak{n})$ at a cusp $\kappa_1$ lying above $\kappa,$ we have 
\begin{equation}
E_{k,\mathfrak{n}}(\chi,z, \kappa) = \sum_{\gamma = 
	\left( \begin{smallmatrix}
	a & b \\ c & d 
	\end{smallmatrix}  \right)
	\in \Gamma_1(\mathfrak{n}) \backslash \Gamma_0(\mathfrak{n}) }  \chi(\langle d\rangle) E_{k,\mathfrak{n}}^{\Gamma_1}(\chi,z, \kappa_1) .
\end{equation}

We can describe the Eisenstein series explicitly at the cusp $\kappa=\infty.$ This cusp corresponds to the trivial ideal class so all the ideals are principal. Two matrices in $SL_2(\ring)$ are left-equivalent under $\Gamma_\infty$ if and only if they have the same bottom row $(c,d)\in \ring \times \ring.$ Fix a principal ideal $\mathfrak{a}$ which is not necessarily integral. A tuple $(c,d)$ is the bottom row of a matrix in $\Gamma_0(\mathfrak{n})$ or $\Gamma_1(\mathfrak{n})$ precisely when $c,d \in \mathfrak{a},$ $\gcd \left(\frac{c}{\mathfrak{a}},\frac{d}{\mathfrak{a}} \right)=1,$ and the tuple satisfies the congruence condition (\ref{eq10}) and (\ref{eq11}) of the respective subgroup (\cite{BMP} pgs. 706 Equation (46)). Since $\mathfrak{a}=\langle \alpha \rangle$ is principal, every element $a \in \mathfrak{a}$ can be written as $a=a^*\alpha,$ and we use the notation $\frac{a}{\mathfrak{a}}$ to denote $a^*.$ Choose a set of representatives $a_2 \in \mathfrak{a}$ which runs over all the congruence classes modulo $\mathfrak{na}.$ Now the sum in Equation (\ref{eq12}) at the cusp $\kappa=\infty$ can be written as     
\begin{align} \label{eq15}
E_{k,\mathfrak{n}}(\chi,z)\coloneqq E_{k,\mathfrak{n}}(\chi,z, \infty) &= \sum_{a_2 \in \left(\mathfrak{a} / \mathfrak{na} \right)^{\times}}  \chi(\langle a_2 \rangle) 
\sum_{\substack{\{c,d\} \in \mathfrak{a}\times \mathfrak{a} \\ \gcd \left(\frac{c}{\mathfrak{a}},\frac{d}{\mathfrak{a}} \right)=1 \\ \mathfrak{n} | c \\ d \equiv a_2 \bmod \mathfrak{n}\mathfrak{a} }} \frac{N(\mathfrak{a})^k}{(cz+d)^{k}} \\ 
&= \sum_{a_2 \in \left(\mathfrak{a} / \mathfrak{na} \right)^{\times}}  \chi(\langle a_2 \rangle) \sum_{\substack{\{c,d\} \in \mathfrak{a}\times \mathfrak{a} \\ \gcd \left(\frac{c}{\mathfrak{a}},\frac{d}{\mathfrak{a}} \right)=1 \\d \equiv a_2 \bmod \mathfrak{n}\mathfrak{a}}} \frac{N(\mathfrak{a})^k}{(\mathfrak{n}cz+d)^{k}} .
\end{align} 
The notation $\{c,d\}$ means that two tuples $(c,d)$ and $(c',d')$ are identified together if there exists $\varepsilon \equiv 1 \bmod \mathfrak{n}$ such that $c = \varepsilon c' $ and $d = \varepsilon d'.$

This definition does not depend on the ideal $\mathfrak{a}$ that was chosen. Given another ideal $\mathfrak{a}'$ in the same ideal class as $\mathfrak{a},$ there exists an element $\alpha \in \field^*$ such that $\mathfrak{a}'=\alpha \mathfrak{a}.$ This implies that the numerator and denominator of the inner sum in (\ref{eq15}) would both be multiplied by $N(\alpha)^k,$ and hence the expression is invariant (\cite{CoSt} pg. 662). 

We now want to describe the Eisenstein series at a cusp $\kappa \neq \infty,$ which corresponds to a (fixed) non-trivial ideal class. The idea is to translate the cusp $\kappa$ to $\infty$ where an explicit expression of the Eisenstein series is easy to describe. 

By Equation (\ref{eq7}), the Fourier expansion of $E_{k,\mathfrak{n}}(\chi,z, \kappa)$ at $\kappa$ is equal to the Fourier expansion of $E_{k,\mathfrak{n}}(\chi,z) |_k\, A^{-1}=(-\alpha_3z+\alpha_1)^{-k}E_{k,\mathfrak{n}}(\chi,A^{-1}z)$ at $\infty.$ Since $\det A = 1,$ a direct computation shows that $(-\alpha_3z+\alpha_1)= (\alpha_3A^{-1}z+\alpha_4)^{-1}.$ Hence, we have
\begin{equation} \label{neededES}
E_{k,\mathfrak{n}}(\chi,z, \kappa) = (\alpha_3A^{-1}z+\alpha_4)^k\sum_{a_2 \in \left(\mathfrak{a} / \mathfrak{na} \right)^{\times}}  \chi(\langle a_2 \rangle)
\sum_{\substack{\{c,d\} \in \mathfrak{a}\times \mathfrak{a} \\ \gcd \left(\frac{c}{\mathfrak{a}},\frac{d}{\mathfrak{a}} \right)=1 \\d \equiv a_2 \bmod \mathfrak{n}\mathfrak{a}}} \frac{N(\mathfrak{a})^k}{(\mathfrak{n}cA^{-1}z+d)^{k}}.
\end{equation}
\section{Fourier Coefficients of Certain Eisenstein Series} \label{FCES}
We are interested in calculating the non-constant Fourier coefficients of (\ref{neededES}). Indeed, this will be the main calculation needed to construct our $p$-adic measure $\lambda.$ We start by defining a more general Eisenstein series by 
\begin{equation} \label{ES}
G_k(z,a_1,a_2,\mathfrak{n}, \mathfrak{a}) = \sum_{\substack{\{c,d\}^+ \in \mathfrak{a}\times \mathfrak{a} \\ (c,d) \equiv (a_1,a_2) \bmod \mathfrak{n}\mathfrak{a}}}\frac{N(\mathfrak{a})^k}{(cz+d)^{k}}.
\end{equation}
The notation $\{c,d\}^+$ means that two tuples $(c,d)$ and $(c',d')$ are identified together if there exists a totally positive $\varepsilon \equiv 1 \bmod \mathfrak{n}$ such that $c = \varepsilon c' $ and $d = \varepsilon d'.$
\begin{lem} \label{reqlem}
	We have
	\begin{equation*}
	\sum_{\substack{\{c,d\} \in \mathfrak{a}\times \mathfrak{a} \\ \gcd \left(\frac{c}{\mathfrak{a}},\frac{d}{\mathfrak{a}} \right)=1 \\d \equiv a_2 \bmod \mathfrak{n}\mathfrak{a}}} \frac{N(\mathfrak{a})^k}{(cz+d)^{k}} = 
	\sum_{a_1 \in \mathfrak{a}/\mathfrak{a}\mathfrak{n}} \sum_{i=1}^{h(\mathfrak{n})} \sum_{ \mathfrak{t} \in C_i}\frac{\mu(\mathfrak{t})} {N(\mathfrak{t})^k} G_k(z,\tau a_1,\tau a_2,\mathfrak{n}, \mathfrak{a}\mathfrak{t})
	\end{equation*}
	where $\mu(\mathfrak{m})$ denotes the Mobius function for any integral ideal $\mathfrak{m}$ given by 
	\begin{equation}
	\mu(\mathfrak{m})=
	\begin{cases}
	1 &\text{if } \mathfrak{m} = \mathcal{O} \\
	(-1)^r &\text{if } \mathfrak{m} = \mathfrak{p}_1 \cdots \mathfrak{p}_r \\
	0 &\text{if } \mathfrak{p}^2 | \mathfrak{m} \text{ for some prime ideal } \mathfrak{p} 
	\end{cases}
	\end{equation}
	and $h(\mathfrak{n})$ denotes the strict ideal class number of $\field \bmod \mathfrak{n}.$
\end{lem}
\begin{proof}
	Recall the following property of the Mobius function:
	\begin{equation}
	\sum_{\mathfrak{t}|\mathfrak{n}} \mu(\mathfrak{t}) = 
	\begin{cases}
	1 &\text{if }\mathfrak{n}=\mathcal{O} \\
	0 &\text{otherwise}
	\end{cases}.
	\end{equation} 
	We can use this property to remove the coprimeness condition. Recall that $\frac{c}{\mathfrak{a}}, \frac{d}{\mathfrak{a}}\in \ring,$ which implies that the ideal $\gcd \left(\frac{c}{\mathfrak{a}},\frac{d}{\mathfrak{a}} \right)$ is an integral ideal. We now have 
	\begin{align}
	\sum_{\substack{\{c,d\} \in \mathfrak{a}\times \mathfrak{a} \\ \gcd \left(\frac{c}{\mathfrak{a}},\frac{d}{\mathfrak{a}} \right)=1 \\ (c,d) \equiv (a_1,a_2) \bmod \mathfrak{n}\mathfrak{a}}}\frac{N(\mathfrak{a})^k}{(cz+d)^{k}} &= 
	\sum_{\substack{\{c,d\} \in \mathfrak{a}\times \mathfrak{a} \\ (c,d) \equiv (a_1,a_2) \bmod \mathfrak{n}\mathfrak{a}}}\frac{N(\mathfrak{a})^k}{(cz+d)^{k}} \sum_{ \mathfrak{t} | \gcd (\frac{c}{\mathfrak{a}},\frac{d}{\mathfrak{a}})} \mu(\mathfrak{t}) \\
	&= \sum_{\gcd(\mathfrak{t},\mathfrak{n})=1}\mu(\mathfrak{t}) \sum_{\substack{\{c,d\} \in \mathfrak{a}\times \mathfrak{a} \\ (c,d) \equiv (a_1,a_2) \bmod \mathfrak{n}\mathfrak{a} \\ (c,d) \equiv (0,0) \bmod \mathfrak{a}\mathfrak{t} }}\frac{N(\mathfrak{a})^k}{(cz+d)^{k}}.
	\end{align}
	Let $C_1, \cdots , C_{h(\mathfrak{n})}$ denote the strict ideal classes of $K$ modulo $\mathfrak{n}.$ Choose integral ideals $\mathfrak{t}_1, \cdots , \mathfrak{t}_{h(\mathfrak{n})}$ such that $\mathfrak{t}_i \in C_i^{-1}$ for all $i \in \{1,\cdots h(\mathfrak{n}) \}$. For each $\mathfrak{t} \in C_i,$ let $(\tau)=\mathfrak{t}\mathfrak{t}_i$ where $\tau \equiv 1 \bmod \mathfrak{n}.$ Then the previous sum is equal to
	\begin{align}
	&\sum_{i=1}^{h(\mathfrak{n})} \sum_{\substack{ \mathfrak{t} \in C_i\\ \gcd(\mathfrak{t},\mathfrak{n})=1}}\mu(\mathfrak{t}) \sum_{\substack{\{c,d\} \in \mathfrak{a}\times \mathfrak{a} \\ (c,d) \equiv (\tau a_1,\tau a_2) \bmod \mathfrak{n}\mathfrak{a} \mathfrak{t} }}\frac{N(\mathfrak{a})^k}{(cz+d)^{k}} \\ 
	=&\sum_{i=1}^{h(\mathfrak{n})} \sum_{ \mathfrak{t} \in C_i}\frac{\mu(\mathfrak{t})} {N(\mathfrak{t})^k} G_k(z,\tau a_1,\tau a_2,\mathfrak{n}, \mathfrak{a}\mathfrak{t}).
	\end{align}
	Summing up over the equivalence classes of $a_1$ modulo $\mathfrak{na},$ we get the desired result.
\end{proof}
The Fourier expansion of the Eisenstein series in Equation (\ref{ES})  has been computed by Klingen (\cite{Klingen}, pg. 181, 182) as follows:  
\begin{thm} \label{fcthm}
	The Fourier expansion of $G_k(z,a_1,a_2,\mathfrak{n}, \mathfrak{a})$ at a general cusp $\kappa$ is given by 
	\begin{multline}
	\delta \left(\frac{(a_1^*)}{\mathfrak{n}\mathfrak{ab}}\right)  N\left(\mathfrak{a}\right)^k \sum_{\substack{d\equiv a_2^* \bmod \mathfrak{n}\mathfrak{ab}^{-1}\\ \{d\}^+ }} \frac{\text{sgn}^k N(d)}{|N(d)|^{k}} \\ 
	+ \frac{(-2\pi i)^{kr}N(\mathfrak{a})^{k-1}N(\mathfrak{b})}{(k-1)!^r  N(\mathfrak{n})|D|^{1/2}} \sum_{\substack {c\equiv a_1^* \bmod \mathfrak{n}\mathfrak{ab} \\ \nu \equiv 0 \bmod \mathfrak{b} /  \mathfrak{n}\mathfrak{a}\mathfrak{d} \\ c\nu \gg 0 ,  \{c\}^+}} \text{sgn} N(\nu) N(\nu)^{k-1}\exp{(a_2^*\nu+c\nu z)}
	\end{multline}
	where $a_1^*=\alpha_4 a_1-\alpha_1 a_2 \in \mathfrak{ab}$ and $a_2^*=\alpha_1 a_2 - \alpha_2 a_1 \in \mathfrak{ab}^{-1}$.
\end{thm}
We use this Fourier expansion to prove the following:
\begin{thm} \label{fourier}
	The Fourier expansion of the partial Eisenstein series $$(\alpha_3A^{-1}z+\alpha_4)^k 
	\sum_{\substack{\{c,d\} \in \mathfrak{a}\times \mathfrak{a} \\ \gcd \left(\frac{c}{\mathfrak{a}},\frac{d}{\mathfrak{a}} \right)=1 \\d \equiv a_2 \bmod \mathfrak{n}\mathfrak{a}}} \frac{N(\mathfrak{a})^k}{(\mathfrak{n}cA^{-1}z+d)^{k}}$$ at the cusp $\infty$ is given by
	\begin{multline*}
	\sum_{l \equiv a_2^* \bmod \mathfrak{n}\mathfrak{ab}^{-1} } \frac{N(\mathfrak{a})^k}{N(l)^k}\sum_{\substack{\mathfrak{t} \in K \\ \mathfrak{t} \mid \langle l \rangle \\ \mathfrak{n} \nmid \mathfrak{t}}} \mu(\mathfrak{t}) \\ 
	+\frac{(-2\pi i)^{kr}N(\mathfrak{b})^k}{(k-1)!^r N(\mathfrak{n})^k|D|^{\frac{2k-1}{2}}} \sum_{i=1}^{h(\mathfrak{n})} \sum_{ \mathfrak{t} \in C_i}\frac{\mu(\mathfrak{t})} {N(\mathfrak{t})^k} \sum_{\substack {\nu \in \ring \\ c\nu \gg 0 ,  \{c\}^+}} \text{sgn} N(\nu) N(\nu)^{k-1}\exp\left(\frac{\mathfrak{t}_i a_2'\nu \mathfrak{b}}{\mathfrak{n} \mathfrak{d}}+\frac{c \nu  \mathfrak{b}^2z}{\mathfrak{d}}\right).
	\end{multline*}
\end{thm}
\begin{proof}
	Lemma \ref{reqlem} implies that we need to calculate the Fourier expansion of 
	\begin{equation} \label{equation1}
	(\alpha_3A^{-1}z+\alpha_4)^k 
	\sum_{a_1 \in \mathfrak{a}/\mathfrak{a}\mathfrak{n}} \sum_{i=1}^{h(\mathfrak{\mathfrak{n}})} \sum_{ \mathfrak{t} \in C_i}\frac{\mu(\mathfrak{t})} {N(\mathfrak{t})^k} G_k(\mathfrak{n}A^{-1}z,\tau a_1,\tau a_2,\mathfrak{n}, \mathfrak{a}\mathfrak{t}).
	\end{equation}
	Theorem \ref{fcthm} along with Equation (\ref{eq7}) now show that the constant term of (\ref{equation1}) is non-zero only when $ a_1^* \equiv 0 \bmod \mathfrak{n}\mathfrak{atb}.$ This implies that the contribution from the terms of the outermost sum are all zero except for one equivalence class of $a_1.$ The nonzero term remaining is 
	\begin{equation}
	\sum_{i=1}^{h(\mathfrak{\mathfrak{n}})} \sum_{ \mathfrak{t} \in C_i}\frac{\mu(\mathfrak{t})} {N(\mathfrak{t})^k} N\left(\mathfrak{at}\right)^k \sum_{\substack{d\equiv \tau a_2^* \bmod \mathfrak{n}\mathfrak{atb}^{-1}\\ \{d\}^+ }} \frac{\text{sgn}^k N(d)}{|N(d)|^{k}}.
	\end{equation}
	The ideals $\mathfrak{n}$ and $\mathfrak{t}$ being coprime, we can rewrite $d\equiv \tau a_2^* \bmod \mathfrak{n}\mathfrak{atb}^{-1}$ as two equivalences modulo $\mathfrak{n}\mathfrak{ab}^{-1}$ and $\mathfrak{atb}^{-1}$. We now obtain the constant term to be
	\begin{equation}
		\sum_{i=1}^{h(\mathfrak{\mathfrak{n}})} \sum_{ \mathfrak{t} \in C_i}\frac{\mu(\mathfrak{t})} {N(\mathfrak{t})^k} \sum_{\substack{d\equiv a_2^* \bmod \mathfrak{n}\mathfrak{ab}^{-1}\\ d \equiv 0 \bmod \mathfrak{ab}^{-1}\mathfrak{t} \\ \{d\}^+ }} N\left(\mathfrak{at}\right)^k \frac{\text{sgn}^k N(d)}{|N(d)|^{k}}. 	\end{equation}
	Switching the order of summation and writing $\langle l \rangle$ for the ideal generated by $l,$ the constant term becomes  
	\begin{align}
	&\sum_{\substack{\mathfrak{t} \in K \\ \mathfrak{p} \nmid \mathfrak{t}}}\frac{\mu(\mathfrak{t})} {N(\mathfrak{t})^k} \sum_{\substack{\mathfrak{t} d\equiv a_2^* \bmod \mathfrak{n}\mathfrak{ab}^{-1}\\ \{d\}^+ }} N\left(\mathfrak{a}\right)^k \frac{1}{N(d)^{k}} \\
	&=\sum_{l \equiv a_2^* \bmod \mathfrak{n}\mathfrak{ab}^{-1} } \frac{N(\mathfrak{a})^k}{N(l)^k}\sum_{\substack{\mathfrak{t} \in K \\ \mathfrak{t} \mid \langle l\rangle\\ \mathfrak{n} \nmid \mathfrak{t}}} \mu(\mathfrak{t}).
	\end{align}
	We consider next the non-constant terms. Note that as $a_1$ varies over the equivalence classes of $\mathfrak{a}/\mathfrak{an},$ $a_1^*$ varies over the equivalence classes of $\mathfrak{ab}/\mathfrak{abn}.$ It is clear then that the non-constant terms are
	\begin{multline}
	\frac{(-2\pi i)^{kr}N(\mathfrak{b})}{(k-1)!^r  N(\mathfrak{n})|D|^{1/2}}\sum_{a_1^* \in \mathfrak{ab}/\mathfrak{ab}\mathfrak{n}} \sum_{i=1}^{h(\mathfrak{n})} \sum_{ \mathfrak{t} \in C_i}\frac{\mu(\mathfrak{t})} {N(\mathfrak{t})^k}N(\mathfrak{at})^{k-1} \\
	\sum_{\substack {c\equiv \tau a_1^* \bmod \mathfrak{n}\mathfrak{ab}\mathfrak{t} \\ \nu \equiv 0 \bmod \mathfrak{b} /  \mathfrak{n}\mathfrak{a}\mathfrak{t}\mathfrak{d} \\ c\nu \gg 0 ,  \{c\}^+}} \text{sgn} N(v) N(\nu)^{k-1}\exp(\tau a_2^*\nu+ c \nu \mathfrak{n} z).
	\end{multline}
	As before, we can split the equivalence modulo $\mathfrak{n}\mathfrak{ab}\mathfrak{t}$ into two equivalences modulo $\mathfrak{n}\mathfrak{ab}$ and $\mathfrak{ab}\mathfrak{t}.$ Moreover, we make a change of variables $\nu \to \frac{\nu \mathfrak{b}}{\mathfrak{n}\mathfrak{adt}}$ and $c \to \mathfrak{abt}c.$ Since $N(\mathfrak{d})=|D|,$ we get
	\begin{equation} 
	\frac{(-2\pi i)^{kr}N(\mathfrak{b})^k}{(k-1)!^r N(\mathfrak{n})^k|D|^{\frac{2k-1}{2}}} \sum_{i=1}^{h(\mathfrak{n})} \sum_{ \mathfrak{t} \in C_i}\frac{\mu(\mathfrak{t})} {N(\mathfrak{t})^k} \sum_{\substack {\nu \in \ring \\ c\nu \gg 0 ,  \{c\}^+}} \text{sgn} N(\nu) N(\nu)^{k-1}\exp\left(\frac{\mathfrak{t}_i a_2'\nu \mathfrak{b}}{\mathfrak{n} \mathfrak{d}}+\frac{c \nu  \mathfrak{b}^2z}{\mathfrak{d}}\right).
	\end{equation}
\end{proof}
\section{Constructing a p-adic Measure} \label{construction}
\subsection{p-adic Measures on Galois Groups} 
Let $\mathfrak{b}_1, \cdots, \mathfrak{b}_h$ represent the strict ideal classes of $K,$ where each $\mathfrak{b}_j$ is the ideal generated by the lower row of the matrix $A_j$ associated with the distinct cusps $\kappa_j.$ By construction, all the ideals $\mathfrak{b}_j$ are relatively prime to $\mathfrak{p}.$

We can decompose $G=G\left(K(p^\infty)/K\right)$ as $\sqcup_j\mathfrak{b}_j^{-1}(\ring_{\mathfrak{p}}^*/U),$ where $U$ is the closure in $\ring_{\mathfrak{p}}^*$ of the subgroup of all totally positive units in $\ring^*$ (\cite{Hida}, Section 3.9). Then for any continuous function $\phi$ on $G$, we can define a continuous function $\phi_j$ on $\ring_{\mathfrak{p}}$ by
\begin{equation}
\phi_j(x)= \begin{cases}
\phi(\mathfrak{b}_j^{-1}x) &\text { if }x \in \ring_{\mathfrak{p}}^*  \\
0 &\text{ otherwise.}
\end{cases}
\end{equation}
We then define p-adic measures on $G$ by setting
\begin{equation}
\int_G \phi \, d\mu = \sum_{j} \int_{\ring_{\mathfrak{p}}} \phi_j d\mu_j.
\end{equation}
\subsection{The Main Construction}
The structure of the Galois group $G$ allows us to define a distribution $\lambda$ on $G$ by defining it on each piece $\mathfrak{b}_j^{-1}(\ring_{\mathfrak{p}}^*/U).$ We will denote the restriction of $\lambda$ to each $\mathfrak{b}_j^{-1}(\ring_{\mathfrak{p}}^*/U)$ by $\lambda_j.$

Fix a strict ideal class $\mathfrak{b}_j.$  
Define a map 
\begin{equation}
\varepsilon_{k,\primep^m, \idealb_j}(\ideala): \ring/\primep^m\ring \to \mathcal(M_k(\Gamma_1(\primep^m)))
\end{equation}
by
\begin{equation}
\varepsilon_{k,\primep^m, \idealb_j}(\ideala) \coloneqq  (\alpha_3A^{-1}z+\alpha_4)^k \sum_{\substack{\{c,d\} \in \mathfrak{a}\times \mathfrak{a} \\ \gcd \left(\frac{c}{\mathfrak{a}},\frac{d}{\mathfrak{a}} \right)=1 \\ d \equiv a_2 \bmod \mathfrak{p}^m\mathfrak{a}}} \frac{N(\mathfrak{a})^k}{(\mathfrak{p}^mcA^{-1}z+d)^{k}}. 
\end{equation}
Let $\mathcal{C}_n$ denote the map which sends a holomorphic Hilbert modular form to the coefficient of $e^{2\pi i Tr(nz)}$ in its Fourier expansion. Theorem \ref{fourier} states that
\begin{multline}
\mathcal{C}_{\primep^m}(\varepsilon_{k,\primep^m, \idealb_j})=\frac{(-2\pi i)^{kr}N(\mathfrak{b}_j)^k}{(k-1)!^r|D|^{\frac{2k-1}{2}}} \sum_{u=0}^m N(\mathfrak{p})^{u(k-1)-mk} \\ \sum_{i=1}^{h(\mathfrak{p}^m)} \sum_{\mathfrak{t} \in C_i}\frac{\mu(\mathfrak{t})}{N(\mathfrak{t})^{k}}\exp\left(\frac{\mathfrak{t}_i a_2'\mathfrak{p}^u \mathfrak{b}_j}{\mathfrak{p}^m \mathfrak{d}}\right).
\end{multline}
Define a distribution on $\mathfrak{b}_j^{-1}(\ring_{\mathfrak{p}}^*/U)$ by
\begin{equation} \label{eq13}
\lambda_j(\mathfrak{a}+\mathfrak{p}^m\ring_{\mathfrak{p}})=  \frac{1}{2^rN(\mathfrak{b}_j)^k} \mathcal{C}_{\primep^m}(\varepsilon_{k,\primep^m, \idealb_j})+ \gamma(k)\lambda_{Haar}
\end{equation}
where
\begin{equation}
\gamma(k)=
\begin{cases}
\frac{N(\mathfrak{p})^{2k-1}}{N(\mathfrak{p})^k-1}\frac{(1-N(\mathfrak{p})^{k-1})^{-1}}{\zeta(1-k)} & \text{ for k even}\\
0 & \text{ otherwise}
\end{cases} . 
\end{equation}
Note that since $a_2' \in \mathfrak{a}\mathfrak{b}_j^{-1},$ the term inside the exponential function, and hence the definition of $\lambda,$ does not depend on the ideals $\mathfrak{b}_j.$ Henceforth we will write $\mathfrak{a}$ instead of $a_2'\mathfrak{b}_j.$  
\begin{thm}
	Let $\chi$ be a Hecke character of finite type with conductor $\mathfrak{p}^m$ for some $m \geq 1.$ Let $k$ be an integer of the same parity as $\chi.$ Then we have
	\begin{equation} \label{eq5}
	\int_G \chi^{-1} (x) \, d\lambda (x) = \frac{h^+\left(1-\chi(\mathfrak{p})N(\mathfrak{p})^{k-1}\right)^{-1}}{L(1-k, \chi)}
	\end{equation}
	where $h^+$ denotes the strict class number of $K$.
\end{thm}
\begin{proof}
	First, we consider the case when $\chi$ is the trivial character and k is even. By finite additivity, we have
	\begin{align}
	\lambda_j(\ring_\mathfrak{p}^*)&=\sum_{\mathfrak{a} \in \ring/\mathfrak{p}\ring} \lambda_j(\mathfrak{a}+\mathfrak{p}\ring_{\mathfrak{p}}) \\
	&=\sum_{\mathfrak{a} \in \ring/\mathfrak{p}\ring}\frac{1}{2^r} \mathcal{C}_{\primep}(\varepsilon_{k,\primep, \idealb_j})+ \sum_{\mathfrak{a} \in \ring/\mathfrak{p}\ring}\gamma(k)\lambda_{Haar}.
	\end{align}
	We calculate each sum separately and start with the first sum. Since the sum of roots of unity modulo $\primep$ is 0, we have 
	\begin{equation}
	 \sum_{\mathfrak{a} \in \ring/\mathfrak{p}\ring}\exp\left(\frac{\mathfrak{t}_i \mathfrak{a}}{\mathfrak{p} \mathfrak{d}}\right) = -1.
	\end{equation}
	Moreover, $\exp\left(\frac{\mathfrak{t}_i \mathfrak{a}}{\mathfrak{d}}\right)= e^{2\pi i Tr(\mathfrak{t}_i \mathfrak{a}/\mathfrak{d})}=1$ for our choice of $\mathfrak{t}_i$ and  $\mathfrak{a},$ so
	\begin{equation}
	\sum_{\mathfrak{a} \in \ring/\mathfrak{p}\ring}\exp\left(\frac{\mathfrak{t}_i \mathfrak{a}}{\mathfrak{d}}\right)= N(\mathfrak{p})-1.
	\end{equation}
	These comments directly imply that the first sum is equal to	
	\begin{align}
	&\frac{(-2\pi i)^{kr}}{2^r(k-1)!^r|D|^{\frac{2k-1}{2}}} \sum_{\substack{\mathfrak{t}\in K \\ \mathfrak{p} \nmid \mathfrak{t}}}\frac{\mu(\mathfrak{t})}{N(\mathfrak{t})^{k}}  \sum_{u=0}^1 N(\mathfrak{p})^{u(k-1)-k} \sum_{\mathfrak{a} \in \ring/\mathfrak{p}\ring}\exp\left(\frac{\mathfrak{t}_i \mathfrak{p}^u \mathfrak{a}}{\mathfrak{p} \mathfrak{d}}\right) \\
	&= \frac{(-2\pi i)^{kr}}{2^r(k-1)!^r |D|^{\frac{2k-1}{2}}}  \sum_{\substack{\mathfrak{t}\in K \\ \mathfrak{p} \nmid \mathfrak{t}}}\frac{\mu(\mathfrak{t})}{N(\mathfrak{t})^{k}} \left( - N(\mathfrak{p})^{-k} + \frac{N(\mathfrak{p})-1}{N(\mathfrak{p})}\right).
	\end{align}
	The term $\sum_{\substack{\mathfrak{t}\in K \\ \mathfrak{p} \nmid \mathfrak{t}}}\mu(\mathfrak{t})N(\mathfrak{t})^{-k}$ is the reciprocal of the Dedekind zeta function with the Euler factor at $\mathfrak{p}$ removed, so the previous sum is equal to 
	\begin{equation}
	\frac{(-2\pi i)^{kr}}{2^r(k-1)!^r |D|^{\frac{2k-1}{2}}} \frac{(1-N(\mathfrak{p})^{-k})^{-1}}{\zeta(k)} \left( - N(\mathfrak{p})^{-k} + \frac{N(\mathfrak{p})-1}{N(\mathfrak{p})}\right).
	\end{equation}
	The functional equation of the Dedekind zeta function is given by
	\begin{equation}
	\frac{1}{\zeta(1-k)}=\frac{(-2\pi i )^{kr}}{2^r(k-1)!^r |D|^{\frac{2k-1}{2}}} \frac{1}{\zeta(k)}.
	\end{equation} 
	It directly implies the first sum is equal to
	\begin{equation} \label{eq1}
	\frac{(1-N(\mathfrak{p})^{-k})^{-1}}{\zeta(1-k)} \left( \frac{N(\mathfrak{p})^k - N(\mathfrak{p})^{k-1}-1}{N(\mathfrak{p})^k}\right). 
	\end{equation}
	The second sum is 
	\begin{equation} \label{eq2}
	\frac{N(\mathfrak{p})^{2k-1}}{N(\mathfrak{p})^k-1} \frac{(1-N(\mathfrak{p})^{k-1})^{-1}}{\zeta(1-k)} \frac{N(\mathfrak{p})-1}{N(\mathfrak{p})}.
	\end{equation} 	
	Adding Equations (\ref{eq1}) and (\ref{eq2}) gives
	\begin{equation}
	\lambda_j(\ring_\mathfrak{p}^*) = \frac{(1-N(\mathfrak{p})^{-k})^{-1}}{\zeta(1-k)}.
	\end{equation}
	Summing up over $j$ equivalence classes we get the required result.
	
	Next we consider the case when $\chi$ is a non-trivial Hecke character of finite type with conductor $\mathfrak{p}^m$ and $k$ is an integer of the same parity as $\chi.$ Since $\chi$ is orthogonal to Haar measure, the integral of $\chi$ against $\lambda_{Haar}$ is zero. The integral of $\chi$ against $\lambda_j$ is 
	\begin{multline} \label{eq3}
	\int_{\ring_\mathfrak{p}^*} \chi^{-1} (x) \, d\lambda_j(x) = \frac{(-2\pi i)^{kr}}{(k-1)!^r|D|^{\frac{2k-1}{2}}} \sum_{u=0}^m N(\mathfrak{p})^{u(k-1)-mk} \\ \sum_{i=1}^{h(\mathfrak{p}^m)} \sum_{\mathfrak{t} \in C_i}\frac{\mu(\mathfrak{t})}{N(\mathfrak{t})^{k}} \sum_{\mathfrak{a} \in \left( \ring / \mathfrak{p}^m \right)^\times} \chi (\mathfrak{a})\exp\left(\frac{\mathfrak{t}_i \overline{\mathfrak{a}}\mathfrak{p}^u }{\mathfrak{p}^m \mathfrak{d}}\right).
	\end{multline}
	We define the Gauss sum 
	\begin{equation}
	\tau(\chi)=\sum_{\mathfrak{x} \in \ring / \mathfrak{p}^m} \chi(\mathfrak{x})e^{2\pi i Tr(\mathfrak{x}/\mathfrak{p}^m\mathfrak{d})}
	\end{equation}
	If $\chi$ is primitive and $x \in \mathcal{O},$ then 
	\begin{equation}
	\sum_{\mathfrak{x} \in \ring / \mathfrak{p}^m} \chi(\mathfrak{x})e^{2\pi i Tr(\mathfrak{x}x/\mathfrak{d})}=
	\begin{cases}
	\overline{\chi}(x)\tau(\chi) &\text{ if } \gcd (x,\mathfrak{p}^m)=1 \\
	0 &\text{ otherwise}
	\end{cases} .
	\end{equation} 
	Then the term with $u=0$ is the only non-zero term in Equation (\ref{eq3}), so we have
	\begin{align}
	& \frac{(-2\pi i)^{kr}}{2^r(k-1)!^r N(\mathfrak{p}^m)^{k}|D|^{\frac{2k-1}{2}}} \sum_{i=1}^{h(\mathfrak{p}^m)} \sum_{\mathfrak{t} \in C_i}\frac{\mu(\mathfrak{t})}{N(\mathfrak{t})^{k}} \sum_{\mathfrak{a} \in \left( \ring / \mathfrak{p}^m \right)^\times} \chi (\mathfrak{a})\exp\left(\frac{\mathfrak{t}_i \overline{\mathfrak{a}}}{\mathfrak{p}^m \mathfrak{d}}\right) \\	
	&= \frac{(-2\pi i)^{kr}}{2^r(k-1)!^r N(\mathfrak{p}^m)^{k}|D|^{\frac{2k-1}{2}}} \sum_{\substack{\mathfrak{t}\in K \\ \mathfrak{p} \nmid \mathfrak{t}}} \frac{\mu(\mathfrak{t})}{N(\mathfrak{t})^{k}} \chi(\mathfrak{t})^{-1} \tau ( \overline{\chi} ) \\
	&= \frac{(-2\pi i)^{kr}}{2^r(k-1)!^r N(\mathfrak{p}^m)^{k}|D|^{\frac{2k-1}{2}}} \tau (\overline{\chi}) \frac{1}{L(k,\overline{\chi})}.
	\end{align}  
	The functional equation of the Hecke $L$-function of a totally real field 
	\begin{equation}
	L(k,\chi) = \frac{(2\pi i)^{kr}}{2^r(k-1)!^r |D|^{\frac{2k-1}{2}}N(\mathfrak{p}^m)^{k-1}\tau(\chi^{-1})} L(1-k,\chi^{-1})
	\end{equation}
	along with the relation $\tau(\chi)\tau(\overline{\chi})=\chi(-1)N(\mathfrak{p}^m)$ gives the integral against $\lambda_j$ to be $L(1-k,\chi)^{-1}.$ Summing over the strict ideal classes of $K$ proves the result.
\end{proof}
\section{p-adic Measures of Totally Real Fields} \label{proof}
It remains to show that the distribution $\lambda$ is bounded. It is desirable to prove the boundedness directly by properties of the Eisenstein series. Instead, we show that $p$-adic $L$-function is invertible as an element of the Iwasawa algebra. We first introduce the $p$-adic $L$-function of totally real fields and follow the construction of Deligne and Ribet (\cite{DR}).

Let $\mathfrak{n}$ be a non-zero integral ideal of $\field$ and let $G_{\mathfrak{n}}$ denote the strict ideal class group of $\field$ modulo $\mathfrak{n}$ defined by:  
\begin{equation}
\left\{\text{prime-to-}\mathfrak{n} \text{ fractional ideals of }K\right\} \slash \left\{ \langle \alpha \rangle \mid \alpha \gg 0 \text{ and } \alpha \in 1+ \mathfrak{n}\mathfrak{b}^{-1} \text{ for some ideal } \mathfrak{b} \right\}.
\end{equation}
By Class Field Theory, we can identify $\varprojlim_n G_{\mathfrak{n}p^n}$ with the Galois group $G=G\left(K(\mathfrak{n}p^\infty)/K\right)$ of the largest abelian extension of $\field$ which is unramified at all the finite places of $\field$ prime to $\mathfrak{n}p.$

Let $\chi$ be any $\C$-valued locally constant character of $G_{\mathfrak{n}}.$ We extend the norm map $N:K \rightarrow \Q$ to the unique continuous norm map $G \rightarrow \Z_p^*$ which we again denote by $N$. Siegel (\cite{Siegel}) showed that the values of $L(s,\chi)= \sum_{\gcd(\mathfrak{a},\mathfrak{n})=1}\chi(\mathfrak{a})N(\mathfrak{a})^{-s}$ at negative integers are rational whenever the values of $\chi$ are rational.

We now assume that all the primes lying above $p$ divide $\mathfrak{n}.$ Any locally constant function $\chi: G \to \Q_p$ factorizes through $G_{\mathfrak{n}'}$ for some $\mathfrak{n} \mid \mathfrak{n}',$ so we can define the values $L(1-k,\chi)\in \Q_p$ unambiguously. In particular, for any $c \in G, k \geq 1,$ and $\chi:G_{\mathfrak{n}} \to \Z_p$ the values 
\begin{equation}
\Delta_c(1-k, \chi)=L(1-k,\chi)-N(c)^{k}L(1-k,\chi_c),
\end{equation} 
where $\chi_c$ is the twist of $\chi$ by $c \in G_{\mathfrak{n}},$ are well defined independently of the choice of $\mathfrak{n}$ and define a $p$-adic measure. Explicitly, the following $p$-adic $L$-function exists:
\begin{thm}(\cite{DR} Theorem 8.2)
	Let $\mathfrak{c}$ be an ideal of $K$ which is relatively prime to $p$. Then there exists a $\Z_p$-valued measure $\mu_{1,\mathfrak{c}}$ on $G$ such that for every integer $k \geq 0$ and every locally constant character $\chi, $ we have
	\begin{equation}
	\int_G \chi(x)N(x)^{k-1}\,d\mu_{1,\mathfrak{c}} = (1-\chi(\mathfrak{c})N(\mathfrak{c})^k))L(1-k,\chi).
	\end{equation}
\end{thm}       
\subsection{Invertibility of Measures in the Iwasawa Algebra}

We begin by describing the $p$-adic measures on $G$ as elements of the Iwasawa algebra $\Lambda(G).$ We refer the reader to \cite{CoSu} for more details.  

Let $\mu_{p}$ denote the group of all $p^{th}$ roots of unity and let $K^{cyc}$ be the unique $\Z_p$ extension of $K$ in $K(p^\infty)$. The group $G=Gal(K(p^\infty)/K)$ decomposes as $\Delta \times \Gamma,$ where $\Delta=Gal(K(\mu_p)/K)$ and $\Gamma=Gal(K^{cyc}/K) \cong \Z_p.$  

Let $\omega(x)=\lim_{n \to \infty}x^{p^n}$ be the Teichmuller character defined over $\Z_p.$ $\omega \circ N$ defines a character of order $q=[(K(\mu_p):K]$ on integral ideals of $K$ which we will again denote by $\omega.$ An arbitrary element $x \in G$ factors as $\omega(x) \cdot \omega(x)^{-1}x \in \Delta \times \Gamma.$ 

Any measure $\nu$ on $\Gamma$ can be extended to the rest of $G$ by the relation
\begin{equation}
\nu(aU)=\nu(U)
\end{equation}
for any $a \in \Delta$ and compact open $U \subset \Gamma.$ Since each $\omega^{i}$ is a continuous function on $G,$ the product $(\omega^{i}\nu)(aU)$ is a measure on $G$ satisfying
{\bf }
\begin{equation}
(\omega^{i}\nu)(aU)=a^i(\omega^i\nu)(U).
\end{equation}   
On the other hand, every measure $\mu$ on $G$ can be decomposed as a sum of measures of the form $\omega^i\nu_i:$
\begin{equation}
\mu=\frac{1}{q}\sum_{i=1}^{q}\omega^i\nu_i
\end{equation}
by setting $\omega^i\nu_i(U)=\sum_{a\in\Delta}\omega(a)^{-i}\mu(aU).$

For any continuous homomorphism $\chi: G \rightarrow \C_p^*,$ let $\chi|_{\Gamma}$ denote the restriction of $\chi$ to $\Gamma.$ The values of $\chi(\omega(x))$ are $q^{th}$ roots of unity in $\C_p^*,$ and hence $\chi|_{\Gamma}$ has the form $x \to x^{-j}$ for some $j \bmod q.$
Since $\chi$ transforms under $\Delta$ by $\omega^{-j},$ the integral of $\chi$ against $\mu$ only involves the term for $i=j$ and equals
\begin{equation}
\int_G \chi(x)\, d\mu = \frac{1}{q}\int_G\chi(x)\,d(\omega^j\nu_j)=\int_{\Gamma} \chi|_{\Gamma} \, d\nu_j
\end{equation}
reducing the integration of the character $\chi$ on $G$ to that of $\chi|_{\Gamma}$ on $\Gamma.$

Now assume that $\nu_i \equiv 0$ whenever $i$ is even and that $\nu_i|_\Gamma$ is a unit in the Iwasawa algebra $\Lambda(G)$ whenever $i$ is odd. Since $\Lambda(G) \cong \ring [[T]]$ has a formal power structure (\cite{CN} Section 4), the invertibility condition is equivalent to the constant term
\begin{equation} \label{eq8}
(\omega^{i}\nu_i)(\Gamma)=\sum_{a\in \Delta}\omega(a)^{-i}\mu(a\Gamma)=\int_G \omega(x)^{-i} \, d\mu
\end{equation}
being a $p$-adic unit.
In this case, for each odd value of $i \bmod q$, there exists an inverse measure $\nu_i^{-1}$ such that $\nu_i \ast \nu_i^{-1}$ is equal to the $\delta$-distribution $\delta_1$ at the identity by Iwasawa's isomorphism theorem, or in other words
\begin{equation}
\int_\Gamma f(x)\, d(\nu_i \ast \nu_i^{-1})(x)=\int_\Gamma \int_\Gamma f(xy) \, d\nu_i(x) \, d\nu_i^{-1}(y)=f(1)
\end{equation} 
for any continuous $\C_p$-valued function on $\Gamma.$ In particular, when $i$ is odd, 
\begin{equation}
\int_\Gamma \chi|_\Gamma \, d\nu_i^{-1} = \left(\int_\Gamma \chi|_\Gamma \, d\nu_i\right)^{-1}.
\end{equation}
We can then define an inverse measure $\mu^{-1}$ on $G$ by  
\begin{equation}
\mu^{-1}=\frac{1}{q}\sum_{\substack{1 \leq i \leq q \\ i \text{ odd}}} \omega^i\nu_i^{-1}.
\end{equation}
Moreover, evaluating the inverse measure $\mu^{-1}$ against odd characters $\chi,$
\begin{equation}
\int_G \chi(x) \, d\mu^{-1}= \int_\Gamma \chi|_\Gamma \, d\nu_j^{-1}=\left(\int_\Gamma \chi|_\Gamma \, d\nu_j\right)^{-1}=\left(\int_G \chi(x) \, d\mu\right)^{-1},
\end{equation}
we see that the integrals of $\mu$ and $\mu^{-1}$ are reciprocals of each other.

\subsection{Boundedness of $\lambda$}

We will prove that $\lambda$ is a bounded measure by showing that the $p$-adic measure on totally real fields $\mu_{1,\mathfrak{c}}$ is a unit in the Iwasawa algebra and that the regularization $\mu^*$ of its inverse $ \mu_{1,\mathfrak{c}}^{-1}$ is equal as a distribution to a multiple of $\lambda$.
   
We start by considering the $p$-adic measure on totally real fields $\mu_{1,\mathfrak{c}}$. By Equation (\ref{eq8}), $\nu_i$ is invertible whenever the integral
\begin{equation} \label{eq4}
\int_G \omega^{-i}(x) \, d\mu_{1,\mathfrak{c}}
\end{equation}
is a $p$-adic unit. Using the property 
\begin{equation}
\int_G \chi(x)N(x)^{n-1} \, d\mu_{1,\mathfrak{c}}(x)=\left( 1 - \chi(\mathfrak{c}) N(\mathfrak{c})^{n}\right) L(1-n,\chi),
\end{equation}
the integral in (\ref{eq4}) is equal to $\left( 1 - \omega^{-i}(\mathfrak{c})N(\mathfrak{c})\right) L(0,\omega^{-i})$. Kummer's criterion in the case of totally real fields (\cite{Greenberg}) implies that $p$ divides $L(0,\omega^{-i})$ for any odd $1 \leq i \leq q$ if and only if $p$ divides the class number of $K(e^{2\pi i/p}).$ Since $p$ is regular by assumption, this does not happen, so the $\nu_i$ are $p$-adic units.

Therefore, we conclude that the measure $\mu_{1,\mathfrak{c}}^{-1}$ satisfies the property 
\begin{equation}
\int_G \chi(x)N(x)^{k-1} \, d\mu_{1,\mathfrak{c}}^{-1} = (1-\chi(\mathfrak{c}) N(\mathfrak{c})^{k})^{-1} L(1-k,\chi)^{-1}
\end{equation} 
for any Hecke character $\chi$ and non-negative integer $k$ of the same parity. Moreover, by direct computation, the regularization 
\begin{equation} 
\mu^{*}(U)= \mu_{1,\mathfrak{c}}^{-1}(U)-N(\mathfrak{c}) \mu_{1,\mathfrak{c}}^{-1}(\mathfrak{c}U)
\end{equation} 
satisfies 
\begin{equation} \label{eq6}
\int_G \chi(x)N(x)^{k-1} \, d\mu^{*}=L(1-k,\chi)^{-1}.
\end{equation}

We are now ready to prove that $\lambda$ is a measure. 
\begin{thm}
	The distribution $\lambda$ defined in Equation \ref{eq13} is a $p$-adic measure.
\end{thm}
\begin{proof}
	Define a distribution $\overline{\lambda}$ on $G$ by $\overline{\lambda}(U)=\lambda(U^{-1}).$ The integral in Equation (\ref{eq5}) is then equal to $\int_G \chi(x) \, d\overline{\lambda}.$ Moreover, since $\mu^{*}$ is a measure, so is $N(x)^{k-1}\mu^{*}$ for any $k.$ Comparing Equations (\ref{eq5}) and (\ref{eq6}), we see that $\overline{\lambda}-h^+N(x)^{k-1}\mu^{*}$ vanishes when integrated against any Hecke character of finite type whenever $k \geq 3$. Since these span the space of locally constant functions on $G,$ the distribution $\overline{\lambda}-h^+N(x)^{k-1}\mu^{*}$ is identically zero. This implies that $\overline{\lambda}$ is equal to $h^+N(x)^{k-1}\mu^{*}$ as distributions and so it is bounded. This implies that $\lambda$ is a measure. 
\end{proof}

\bibliographystyle{amsalpha}
\bibliography{references}
\end{document}